\documentclass[reqno, 11pt, a4paper]{amsart}

\usepackage{amsfonts, amsthm, amsmath, amssymb}
\usepackage{hyperref}
\hypersetup{colorlinks=false}
\usepackage{booktabs}
\usepackage[font=small, margin=40pt,figureposition=below]{caption}
\usepackage{graphicx}
\usepackage{tikz-cd}
\usepackage[margin=3.5cm]{geometry}

\numberwithin{equation}{section}

\renewcommand{\phi}{\varphi}
\renewcommand{\rho}{\varrho}

\newcommand{\ZZ}{\mathbb{Z}}

\newcommand{\QQ}{\mathbb{Q}}

\newcommand{\F}{\mathbb{F}}

\renewcommand{\leq}{\leqslant}

\renewcommand{\geq}{\geqslant}

\renewcommand{\b}{\mathbf{b}}

\newcommand{\ve}{\varepsilon}

\newcommand{\uu}{\underline}

\newcommand{\abs}[1]{\left\lvert#1\right\rvert}

\newcommand{\aff}{\operatorname{aff}}

\newtheorem{theorem}{Theorem}[section]
\newtheorem{corollary}[theorem]{Corollary}
\newtheorem{lemma}[theorem]{Lemma}
\newtheorem{conjecture}[theorem]{Conjecture}
\newtheorem{proposition}[theorem]{Proposition}

\theoremstyle{definition}

\newtheorem{remark}[theorem]{Remark}

\begin{document}
	\author{Matteo Verzobio}
	\address{IST Austria\\
		Am Campus 1\\
		3400 Klosterneuburg\\
		Austria}
	
	\email{matteo.verzobio@gmail.com}	
	\title{Counting rational points on smooth hypersurfaces with high degree}
	\begin{abstract}
		Let $X$ be a smooth projective hypersurface defined over $\mathbb{Q}$. We provide new bounds for rational points of bounded height on $X$. In particular, we show that if $X$ is a smooth projective hypersurface in $\mathbb{P}^n$ with $n\geq 4$ and degree $d\geq 50$, then the set of rational points on $X$ of height bounded by $B$ have cardinality $O_{n,d,\varepsilon}(B^{n-2+\varepsilon})$. If $X$ is smooth and has degree $d\geq 6$, we improve the dimension growth conjecture bound. We achieve an analogue result for affine hypersurfaces whose projective closure is smooth.
	\end{abstract}
	\maketitle
	\thispagestyle{empty}
	\setcounter{tocdepth}{1}
	\section{Introduction}
	Let $X$ be a projective geometrically irreducible hypersurface defined over $\QQ$ of degree $d\geq 2$ in $\mathbb{P}^n$. The uniform dimension growth conjecture \cite[Conjecture 2]{HBannals}, stated by Heath-Brown, asserts that 
	\begin{equation}\label{eq:dimgrowth}
		N(X,B)\ll_{n,d,\ve} B^{n-1+\ve},
	\end{equation} 
	where $N(X,B)$ counts the number of rational points on $X$ of height at most $B$. This conjecture is now a theorem \cite[Theorem 0.3]{salberger} by Salberger for $d\geq 4$, or by Heath-Brown \cite[Theorem 2]{HBannals} for $d=2$. The conjecture remains open for $d=3$: however, Salberger proved that it holds if the error term is allowed to depend on $X$ \cite[Theorem 0.1]{salberger}. The aim of this paper is to improve \eqref{eq:dimgrowth} under suitable conditions on $X$. Specifically, our goal is to show that if $X$ is smooth, then $N(X,B)\ll_{n,d,\ve} B^{n-2+\ve}$. This has been proven for $n \geq 28$ by Marmon in \cite[Theorem 1.2]{marmon28} and our objective is to extend the result to $n\geq 4$, subject to additional assumptions on $d=\deg X$.
	
	Let $X\subseteq \mathbb{P}^n$ be a projective variety. Define
	\[
	N(X,B)=\#\left\{\uu{x}=(x_0,x_1,\dots, x_n)\in X(\QQ): x_i\in \ZZ, \begin{array}{l}
		\abs{x_i}\leq B\, \forall \, 0\leq i\leq n,\\ 
		\gcd_{0\leq i\leq n}(x_i)=1\end{array}
	\right\}.
	\]
	We will work with a fixed $B\geq 2$ and we will say that $\uu{x}\in X$ is of bounded height if it belongs to the set above. For $d\geq 6$, define the decreasing positive function
	\begin{equation}\label{eq:theta}
		\theta(d)=\begin{cases}
			0 &\text{ if } d\geq 50,\\
			\frac{11}{4\sqrt[3]{d}}-\frac{3}{4} &\text{ if } 50> d\geq 20,\\
			\frac{3}{\sqrt[3]{d}}+\frac{1}{3\sqrt{d}}-\frac{11}{12}&\text{ if } 20> d\geq 16,\\
			\frac{3}{\sqrt[3]{d}}+\frac{2}{3\sqrt{d}}-1 &\text{ if } 16> d\geq 9,\\
			\frac{2}{\sqrt{d}} &\text{ if } 9> d\geq 6.
		\end{cases}
	\end{equation}
	\begin{theorem}\label{thm:main}
		Let $d\geq 6$ and $n\geq 4$.
		Let $X$ be a smooth projective hypersurface defined over $\QQ$ of degree $d$ in $\mathbb{P}^n$. Then, for all $\ve>0$,
		\[
		N(X,B)\ll_{n,d,\ve} B^{n-2+\theta(d)+\ve} \text{ if } n\geq 5,
		\]
		and 
		\[
		N(X,B)\ll_{d,\ve} B^{2+\max\left\{0,\frac{45}{16\sqrt{d}}-\frac{3}{4}\right\}+\ve} \text{ if } n=4.
		\]
	\end{theorem}
	Notice that, for $d\geq 6$, our bound is strictly better than the one that follows from the uniform dimension growth conjecture, since $\frac{45}{16\sqrt{d}}-\frac{3}{4}<1$ and $\theta(d)< 1$ for $d\geq 6$. In the case when $d<6$, our method does not provide better bounds than the one from the dimension growth conjecture. The main reason is that we are not able to bound efficiently the rational points of bounded height that lie on irreducible curves of degree less than $4$ (see Lemma \ref{lem:remlines}). The following corollary follows on combining Theorem \ref{thm:main} with the aforementioned result by Marmon \cite[Theorem 1.2]{marmon28}.
	\begin{corollary}\label{cor:marm}
		Let $n,d\geq 4$ and let $\ve>0$. Let $X$ be a smooth projective hypersurface defined over $\QQ$ of degree $d$ in $\mathbb{P}^n$. If $d\geq 50$ or $n\geq 28$, then
		\begin{equation}\label{eq:main}
			N(X,B)\ll_{n,d,\ve} B^{n-2+\ve}.
		\end{equation}
	\end{corollary}
	We believe that \eqref{eq:main} should hold for all $n\geq 4$ and $d\geq 3$, while it is easy to show that it cannot hold in general for $d\leq 2$ or $n\leq 3$.
	\begin{conjecture}
		Let $n\geq 4$ and $d\geq 3$. Let $X$ be a smooth projective hypersurface defined over $\QQ$ of degree $d$ in $\mathbb{P}^n$. For all $\ve>0$, 
		\[
		N(X,B)\ll_{n,d,\ve} B^{n-2+\ve}.
		\]
	\end{conjecture}
	\begin{remark}
		The bound of Theorem \ref{thm:main} is optimal (up to ignoring the $\ve$ in the exponent), in the sense that we cannot improve it for all $d\geq6$ and $n\geq 4$. For example, let $X\subseteq \mathbb{P}^5$ be defined by
		\[
		F(x_0,x_1,x_2,x_3,x_4,x_5)=x_0(x_0^d+x_5^d)+x_1(x_1^d+x_4^d)+x_2(x_2^d+x_3^d),
		\]
		that is smooth and contains the plane $x_0=x_1=x_2=0$. Thus, $N(X,B)\gg B^3$.
	\end{remark}
	\begin{remark}\label{rem:23}
		We briefly discuss the assumption $n\geq 4$. To prove \eqref{eq:main} for $n=2$, we would need to show that $N(C,B)\ll_{d,\ve}B^{\ve}$ for a smooth projective plane curve $C$. Since $d\geq 4$, the curve $C$ has genus strictly larger than $1$ and then \[N(C,B)\ll_C 1\] by Falting's Theorem. However, removing the dependence on $C$ and proving that $N(C,B)\ll_{d,\ve}B^{\ve}$ is a major open conjecture. In the case $n=3$, \eqref{eq:main} holds, for $d\geq 9$, if and only if $X$ does not contain any rational line \cite[Theorem 0.5]{salberger}.
	\end{remark}
	
	We also prove an affine analogue of Theorem \ref{thm:main}.
	Let $Y\subseteq \mathbb{A}^n$ be an affine variety. Define
	\[
	N_{\aff}(Y,B)=\#\{\uu{x}=(x_1,\dots, x_n)\in Y(\ZZ): \abs{x_i}\leq B\, \forall \,1\leq i\leq n\}.
	\]
	An affine analogue of the uniform dimension growth conjecture has been proved by Vermeulen in \cite{vermeulen2024dimension}, see \eqref{eq:affdimgrowth} for more details.
	\begin{theorem}\label{thm:mainaff}
		Let $n\geq 5$ and $d\geq 6$.
		Let $Y$ be an affine hypersurface defined over $\QQ$ of degree $d$ in $\mathbb{A}^n$ such that the projective closure of $Y$ is smooth. Then, for all $\ve>0$,
		\[
		N_{\aff}(Y,B)\ll_{n,d,\ve} B^{n-3+\theta(d)+\ve}.
		\]
	\end{theorem}
	\begin{remark}\label{rem:addhyp}
		We briefly discuss the assumption $n\geq 5$. If $n=2,3$, then the bound $N_{\aff}(Y,B)\ll_{d,\ve}B^{n-3+\ve}$ clearly cannot hold, in general. Let $n=4$. Let $F_1(x_0,x_1,x_2)=0$ be a smooth projective curve $C$ of degree $d$ and let $f_1(x_1,x_2)=F_1(1,x_1,x_2)$. Consider the affine variety $Y$ defined by
		\[
		f(x_1,\dots,x_4)=f_1(x_1,x_2)+x_3x_4^{d-1}+x_3^{d-1}x_4.
		\] 
		The projective closure of $Y$ is defined by 
		\[F_1(x_0,x_1,x_2)+x_3x_4^{d-1}+x_3^{d-1}x_4,\] and it is smooth. The hypersurface $f_1(x_1,x_2)=x_3=0$ is contained in $Y$ and then \[N_{\aff}(Y,B)\gg N_{\aff}(\{f_1(x_1,x_2)=0\},B)B.\] In order to have the bound of Theorem \ref{thm:mainaff} for $N_{\aff}(Y,B)$, that is $N_{\aff}(Y,B)\ll_{d,\ve} B^{1+\ve}$, we would need to show that $N_{\aff}(\{f_1(x_1,x_2)=0\},B)\ll_{d,\ve} B^{\ve}$. As we already mentioned, this is an open conjecture.
		In Lemma \ref{rem:counterexample}, we prove an analogue of Theorem \ref{thm:mainaff} for $n=4$.
	\end{remark}
	
	We will prove Theorem \ref{thm:main} and \ref{thm:mainaff} via a modification of Salberger's determinant method \cite[Lemma 3.2]{salberger}. A key feature of this paper is that, rather than applying the determinant method to surfaces, as it is typically done, we apply it to threefolds. This leads to efficient bounds thanks to recent optimal bounds on the cardinality of integral points of bounded height on affine curves of very high degree by Binyamini, Cluckers, and Novikov \cite[Theorem 2]{optimal}. Indeed, when one applies the determinant method to threefolds, one has to bound the cardinality of rational points of bounded height on surfaces of bounded degree, which we will handle using \cite[Section 7]{salberger}, and on curves that may have very high degree. 
	
	The problem of bounding the cardinality of rational points of bounded height on smooth projective hypersurfaces is well-established and has been studied in works such as \cite{BHB}, \cite{BHB2}, \cite{salAENS}, \cite{salberger3fold}, and \cite{sal-wooley}. For the case of degree $3$, see \cite{cubic}. As already mentioned, Marmon in \cite{marmon28} handles the same problem for $n$ large, obtaining the bounds via the study of certain exponential sums. As is often the case, when $n$ is small and $d$ is large, the determinant method yields stronger bounds, and this remains true in our setting. In the case when $d$ is much larger than $n$ (say, roughly speaking, $d\geq n^n$), the approach used in \cite{sal-wooley} by Salberger and Wooley would lead to better bound; see \cite[Corollary 1.2]{sal-wooley}. Both our method and theirs rely on iteratively slicing the variety with hypersurfaces. The main difference is that we perform this slicing procedure using hyperplanes, while they use hypersurfaces that may have higher degree. In the case when $d$ is much larger than $n$, the number of hypersurfaces they use is smaller than the number from our method and so it leads to better results.

	We will adopt the following strategy. If $X$ is a smooth hypersurface in $\mathbb{P}^n$ (or $\mathbb{A}^n$), then by fixing one variable we reduce to studying hypersurfaces in $\mathbb{A}^{n-1}$. Up to some cases that we treat separately, we can assume that the projective closure of these affine hypersurfaces is smooth. We repeat this procedure to reduce to studying smooth hypersurfaces in $\mathbb{A}^4$. Using Salberger's determinant method \cite{salberger}, we show that integral points of bounded height lie on a set of surfaces or on a set of curves. We then bound integral points of bounded height on surfaces or curves using Salberger's determinant method again (\cite[Section 7]{salberger}) or an optimal bound for points on curves (\cite[Theorem 2]{optimal} if the degree of the curve is at least $4$, or \cite[Proof of Corollary 3.9]{salAENS} if the degree of the curve is at most $3$).
	
	In recent years, considerable effort has been devoted to eliminating the $\ve$ in the exponent of the dimension growth conjecture bound \eqref{eq:dimgrowth} and understanding the dependence on $d$ in the error term; see, for instance, the work of Castryck, Cluckers, Dittmann, and Nguyen \cite{castryck}, as well as the aforementioned \cite{optimal}. 
	Similarly, there has been interest in generalizing the dimension growth conjecture bound to the global field setting; see the work of Sedunova \cite{sedunova}, Paredes and Sasyk \cite{paredes}, and Vermeulen \cite{vermeulenFqt}. While we believe that their methods could potentially be adapted to the setting considered in this paper, we do not pursue this direction here.
	
	While this project was in its final stages, Binyamini, Cluckers, and Kato proved in \cite{dsquare} that the dependence on the degree in the implied constant of the dimension growth conjecture bound \eqref{eq:dimgrowth} is quadratic. This result provides an efficient way to bound the number of points of bounded height on hypersurfaces of high degree. With this new tool, it seems reasonable that the bounds established in this paper can be improved. We intend to pursue this direction in future work.

	\subsection*{Acknowledgements}
	The author is very grateful to Tim Browning for suggesting the problem and for many useful discussions. We thank the anonymous referees for their many helpful comments, which improved the exposition of the paper. We are also grateful to Gal Binyamini for their interest in this work and for drawing our attention to the aforementioned paper \cite{dsquare}.
	
	We shared an early version of this paper with Per Salberger, who mentioned that he announced a new bound for smooth threefolds in $\mathbb{P}^4$ during a talk in 2019 (see \cite{salbtalk} for the abstract). This result has not been published.  
	
	While working on this paper,
	the author was supported by the European Union’s Horizon 2020 research and 
	innovation program under the Marie Skłodowska-Curie Grant Agreement No. 
	101034413.
	
	\section{Preliminaries}
	In this section, we collect some results that we are going to use in the following sections.
	
	Given a projective variety $X$ of degree $d$, it is rather easy to show \cite[Theorem 1]{Browning-JRAM} that 
	\begin{equation}\label{eq:trivialproj}
		N(X,B)\ll_d B^{\dim X+1}.
	\end{equation}
	In the same way, given an affine variety $Y$ of degree $d$, we have 
	\begin{equation}\label{eq:trivialaff}
		N_{\aff}(Y,B)\ll_{d}B^{\dim Y}.
	\end{equation}
	We refer to these bounds as the trivial bounds. 
	
	As we pointed out in the introduction, we will need to bound the cardinality of integral points of bounded height on affine curves of very high degree. We will do so using the work of Binyamini, Cluckers, and Novikov in \cite[Theorem 2]{optimal}. We state their result in the form we need in the next lemma.
	\begin{lemma}\label{lem:poncurves}
		Let $C\subseteq \mathbb{A}^n$ be an algebraic curve of degree $d$ such that each irreducible component of $C$ has degree at least $\delta\leq d$. Then, for all $\ve>0$,
		\[
		N_{\aff}(C,B)\ll_{n,\ve}B^\ve\left(dB^{\frac{1}{\delta}}+d^2\right).
		\]
		
		Fix $q$ a positive integer and $\uu{z}=(z_1,z_2,\dots, z_n)\in \mathbb{A}^n(\ZZ)$. Let
		\[
		N_{\aff}(C,B,q,\uu{z})=\#\{\uu{x}\in C(\ZZ)\mid \abs{x_i}\leq B, x_i\equiv z_i\mod q\,\, \forall 1\leq i\leq n\}
		\] 
		and then
		\[
		N_{\aff}(C,B,q,\uu{z})\ll_{n,\ve}B^\ve\left(d\left(\frac{B}{q}\right)^{\frac{1}{\delta}}+d^2\right).
		\]
	\end{lemma}
	\begin{proof}
		Let $\{C_i\}$ be the set of irreducible components of $C$, each of degree $d_i$. So, $\sum d_i=d$ and
		\[
		N_{\aff}(C_i,B)\ll_{n,\ve}d_i^2B^{\frac{1}{d_i}}B^{\frac{\ve}{2}}
		\]
		by \cite[Theorem 2]{optimal}.
		If $d_i\leq 2(\ve)^{-1}$, then
		\[
		d_i^2B^{\frac{1}{d_i}}B^{\frac{\ve}{2}}\ll_\ve B^{\frac{1}{\delta}+\ve}.
		\]
		If $d_i\geq 2(\ve)^{-1}$, then
		\[
		d_i^2B^{\frac{1}{d_i}}B^{\frac{\ve}{2}}\ll_\ve d_i^2B^{\ve}.
		\]
		In both cases,
		\[
		N_{\aff}(C_i,B)\ll_{n,\ve}B^{\ve}(d_i^2+B^{\frac{1}{\delta}})
		\]
		and hence
		\[
		N_{\aff}(C,B)=\sum_i N_{\aff}(C_i,B)\ll_{n,\ve}\sum_i B^{\ve}\left(d_i^2+B^{\frac{1}{\delta}}\right)\ll_{n,\ve}B^\ve\left(dB^{\frac{1}{\delta}}+d^2\right).
		\]
		
		For the second part of the lemma, we can assume $\abs{z_i}\leq q$ for all $1\leq i\leq n$. Let $C'$ be the curve obtained by $C$ after the change of variables $\uu{x}\to (\uu{x}-\uu{z})/q$. Then,
		\[
		N_{\aff}(C,B,q,\uu{z})\leq N_{\aff}(C',1+B/q)
		\]
		and we conclude by applying the first part of the lemma.
	\end{proof}
	To bound the number of integral points of bounded height on affine surfaces, we use Salberger's work in \cite[Section 7]{salberger}. 
	\begin{lemma}\label{lem:surfaces}
		Let $W$ be an irreducible surface defined over $\QQ$ in $\mathbb{A}^4$ of degree $d$ and let $N_{\aff,e}(W,B)$ be the cardinality of integral points of bounded height on $W$ that do not lie on irreducible curves of degree at most $e-1$ on $W$. Then, for all $\ve>0$,
		\[
		N_{\aff,e}(W,B)\ll_{d,e,\ve} B^{\frac{2}{\sqrt{d}}+\ve}+B^{\frac{1}{\sqrt{d}}+\frac 1e+\ve}.
		\] 
	\end{lemma}
	\begin{proof}
		The proof of this lemma follows the second part of the proof of \cite[Theorem 6.1]{salberger}.
		Assume that $W$ is geometrically irreducible.
		By \cite[Lemma 8.1]{salsurf}, there exists $W_1$ in $\mathbb{A}^3$, geometrically integral and of the same degree as $W$, such that \[N_{\aff}(W,B)\ll_{d} N_{\aff}(W_1,O_{d}(B)).\] Moreover, there exist two curves $C\subseteq W$ and $C_1\subseteq W_1$ such that $W\setminus C$ is isomorphic to $W_1\setminus C_1$. Furthermore, the degree of $C$ and $C_1$ is $O_d(1)$. We can bound the integral points of bounded height on irreducible components of $C$ of degree at least $e$ by $O_{d,e,\ve}(B^{\frac{1}{e}+\ve})$ using Lemma \ref{lem:poncurves}. By \cite[Example 18.16]{harris}, the degree of curves on $W$ outside $C$ does not change under the isomorphism. Hence,
		\[
		N_{\aff,e}(W,B)\ll_{d,e,\ve} N_{\aff,e}(W_1,O_{d}(B))+B^{\frac{1}{e}+\ve}.
		\]
		By \cite[Theorem 7.2]{salberger}, the non-singular integral points of bounded height on $W_1$ that do not lie on curves of degree at most $e-1$ are \[
		O_{d,e,\ve}\left(B^{\frac{1}{\sqrt{d}}+\frac{1}{e}+\ve}+B^{\frac{2}{\sqrt{d}}+\ve}\right).
		\] The singular locus of $W_1$ is contained in the union of $O_d(1)$ curves of degree $O_d(1)$. Thus, the singular integral points of bounded height on $W_1$ that do not lie on curves of degree at most $e-1$ are
		\[
		O_{d,e,\ve}\left(B^{\frac{1}{e}+\ve}\right)
		\]
		and the proof is completed, under the assumption that $W$ is geometrically irreducible.
		
		If $W$ is not geometrically irreducible, the integral points on $W$ would lie on a proper subvariety $W_1$ of degree $O_d(1)$ (see \cite[Proof of Theorem 2.1]{salberger}). Then, $\dim W_1\leq 1$ and we easily conclude.
	\end{proof}
	
	Given a projective variety $X$, we denote by $N_e(X,B)$ the number of rational points of bounded height that do not lie on irreducible curves of degree at most $e-1$.
	\begin{lemma}\label{lem:projsurf}
		Let $\ve>0$. Let $X\subseteq \mathbb{P}^n$ be a projective irreducible surface defined over $\QQ$ of degree $d$. Then,
		\[
		N_e(X,B)\ll_{n,d,e,\ve} B^{\frac{3}{\sqrt{d}}+\ve}+B^{\frac{3}{2\sqrt{d}}+\frac{2}{e}+\ve}.
		\]
		
		If $X$ is smooth and $2\leq e \leq d-2$, then 
		\[
		N_e(X,B)\ll_{n,d,\ve} B^{\frac{3}{\sqrt{d}}+\ve}+B^{\frac{2}{e}+\ve}.
		\]
	\end{lemma}
	\begin{proof}
		The first part of the lemma follows from \cite[Corollary 3.22 and Theorem 6.1]{salberger}. For the second part of the lemma, notice $d\geq 4$ and then
		\[
		\frac{3}{2\sqrt{d}}+\frac{2}{d-1}\leq \frac{3}{\sqrt{d}}.
		\]
		By the first part of the lemma,
		\[
		N_{d-1}(X,B)\ll_{n,d,\ve}  B^{\frac{3}{\sqrt{d}}+\ve}+B^{\frac{3}{2\sqrt{d}}+\frac{2}{d-1}+\ve} \ll_{n,d,\ve} B^{\frac{3}{\sqrt{d}}+\ve}.
		\]
		As proved by Colliot-Thélène in \cite[Appendix]{HBannals}, there are $O_d(1)$ irreducible curves of degree $\leq d-2$ on $X$ and the lemma follows. 
	\end{proof}
	\begin{lemma}\label{lem:NF}
		Let $X$ be a smooth projective hypersurface of degree $d$ in $\mathbb{P}^n$ for $n\geq 4$. There are no varieties of dimension $n-2$ and degree less than $d$ lying on $X$.
	\end{lemma}
	\begin{proof}
		This is a corollary of the Noether–Lefschetz Theorem, see \cite[Theorem 2.3]{sal-wooley}.
	\end{proof}
	We say that an affine variety $W\subset \mathbb{A}^n$ of dimension $m$ is cylindrical over a curve if there exists a linear map $\pi:\mathbb{A}^n\to \mathbb{A}^{n-m+1}$ such that $\pi(W)$ is a curve.
	
	If $W$ is an irreducible affine variety defined over $\QQ$, of dimension $m$, degree $d\geq 4$, and not cylindrical over a curve, then
	\begin{equation}\label{eq:affdimgrowth}
		N_{\aff}(W,B)\ll_{n,d,\ve} B^{m-1+\ve}.
	\end{equation}
	This affine version of the uniform dimension growth conjecture has been proved by Vermeulen \cite[Theorem 1.2]{vermeulen2024dimension}.
	
	\begin{lemma}\label{lem:cylsing}
		Let $Y$ be an affine hypersurface in $\mathbb{A}^n$ for $n\geq 5$ of degree $d\geq 2$, such that its projective closure is smooth. Let $W$ be an irreducible subvariety of dimension $n-2$. Then, $W$ cannot be cylindrical over a curve.
	\end{lemma}
	\begin{proof}
		We assume that $W$ is cylindrical over a curve and we aim for a contradiction. After a linear change of variables, we can assume that $W$ is defined by polynomials in $n-\dim W+1=3$ variables. Thus, we assume $W=V(f_i(x_1,x_2,x_3))_{1\leq i\leq s}$. Assume that $Y$ is defined by $f=0$ with 
		\[
		f(x_1,\dots ,x_n)=h_0(x_1,\dots, x_n)+x_n^{d-1}h_1(x_1,\dots, x_n)
		\]
		for $\deg_{x_n}(h_0(x_1,\dots, x_n))\leq d-2$. Since $W\subseteq Y$, then \[ V(f_i(x_1,x_2,x_3))_{1\leq i\leq s}\subseteq V(h_1(x_1,\dots, x_n))\] and $h_1$ is non-zero since otherwise the projective closure of $Y$ would be singular. So, $h_1$ has degree $1$ and after a linear change of variables we can assume $h_1=x_3$ and $W=V(f_0(x_1,x_2),x_3)$. Therefore, 
		\[
		f(x_1,\dots ,x_n)=f_0(x_1,x_2)g_0(x_1,\dots ,x_n)+x_3g_1(x_1,\dots ,x_n)
		\]
		and the projective closure of $Y$ is defined by
		\[
		F(x_0,x_1,\dots ,x_n)=F_0(x_0,x_1,x_2)G_0(x_0,x_1,\dots ,x_n)+x_3G_1(x_0,x_1,\dots ,x_n)=0.
		\]
		Notice that $G_1$ is non-constant since $\deg F=d\geq 2$ and $F_0$ has degree at least $2$, since otherwise $W$ would have degree $1$ contradicting Lemma \ref{lem:NF}.
		A point that satisfies \[x_0=x_1=x_2=x_3=G_1=0\] is singular, contradiction.
	\end{proof}
	
	\begin{lemma}\label{lem:remlines}
		Let $\ve>0$. Let $X$ be a smooth projective hypersurface defined over $\QQ$ of degree $d> n+1$ in $\mathbb{P}^n$ for $n\geq 4$. Let $Z$ be
		the union of all algebraic curves on $X$ with irreducible components of degree at
		most three. The rational points of bounded height on $Z$ are
		\[
		O_{n,d,\ve}\left(B^{n-2+\ve}\right).
		\]
	\end{lemma}
	\begin{proof}
		By \cite[Lemma 6.1]{salAENS}, the reduced scheme structure of $Z$ is a proper closed subvariety of $X$, defined over $\QQ$, of degree $O_d(1)$. If $\dim(Z)\leq n-3$, we conclude by the trivial bound \eqref{eq:trivialproj}. If $\dim(Z)= n-2$, each irreducible component has degree at least $d$ by Lemma \ref{lem:NF} and then we conclude using the bound \eqref{eq:dimgrowth}. Finally, $Z$ cannot have dimension equal to $n-1$ since it is proper.
	\end{proof}
	\begin{lemma}\label{lem:remaininglinesaff}
		Let $\ve>0$. Let $Y$ be an affine hypersurface defined over $\QQ$ of degree $d> n+1$ in $\mathbb{A}^n$ for $n\geq 4$, such that its projective closure is smooth. Let $W$ be
		the union of all algebraic curves on $Y$ with irreducible components of degree at
		most three. If $n\geq 5$, then the integral points of bounded height on $W$ are
		\[
		O_{n,d,\ve}\left(B^{n-3+\ve}\right).
		\]
		If $n= 4$, then 
		the integral points of bounded height on $W$ are
		\[
		O_{d,\ve}\left(B^{1+\frac 1d+\ve}\right).
		\]
	\end{lemma}
	\begin{proof}
		Let $W_1$ be an irreducible component of the reduced scheme structure of $W$. By \cite[Lemma 6.1]{salAENS} we cannot have $\dim(W_1)=n-1$ and if $\dim(W_1)\leq n-3$ we conclude with the trivial bound \eqref{eq:trivialaff}. If $\dim(W_1)=n-2$, then $W_1$ has degree at least $d$ by Lemma \ref{lem:NF} (since otherwise the projective closure of $W_1$ would be a hypersurface of degree less than $d$). If $W_1$ is not cylindrical over a curve, we conclude by \eqref{eq:affdimgrowth}. If $n\geq 5$, then $W_1$ is not cylindrical over a curve by Lemma \ref{lem:cylsing} and so we are done. If $n=4$, we conclude by \cite[Theorem A]{pilaaff}. 
	\end{proof}
	
	\begin{lemma}\label{lem:B}
		Let $F\in \ZZ[x_0,x_1,\dots,x_n]$ be a homogenous non-singular polynomial of degree $d$. There exists a matrix $A\in \operatorname{SL}_{n+1}(\ZZ)$ with $\max{\abs{a_{i,j}}}\ll_{n,d} 1$ with the following properties. Let $G=F\circ A^{-1}$, let $Y_b$ be the affine variety defined by $G(1,x_1,\dots ,b)=0$, and $\mathcal{B}$ be the set of integers $b$ such that $Y_b$ has degree strictly smaller than $d$ or its projective closure is not smooth. Then, the projective variety $G=0$ is smooth, of degree $d$, and $\#\mathcal{B}\ll_{n,d} 1$.
	\end{lemma}
	\begin{proof}
		See Lemma 5 and the proof of Proposition 2 at page 409 of \cite{BHB}. 
	\end{proof}
	\begin{lemma}\label{lem:irr}
		Let $X$ be a smooth hypersurface in $\mathbb{P}^n$ for $n\geq 4$. For any hyperplane $H$, the projective variety $X\cap H$ in $\mathbb{P}^{n-1}$ is irreducible and $\deg X= \deg (X\cap H)$.
	\end{lemma}
	\begin{proof}
		We can assume $H=\{x_0=0\}$ and that $X$ is defined by $F(x_0,x_1,\dots, x_n)=0$ with
		\[
		F(x_0,x_1,\dots,x_n)=x_0G_0(x_0,x_1,\dots, x_n)+G(x_1,\dots, x_n).
		\] 
		Thus, $X\cap H$ is defined by $G(x_1,\dots, x_n)=0$. Notice that $G$ is non-zero, since otherwise $X$ would be reducible, and then $\deg(G)=\deg(F)$. If $X\cap H$ would be reducible, then $G=G_1G_2$ with $G_1,G_2$ non-constant and so 
		\[
		F(x_0,x_1,\dots,x_n)=x_0G_0(x_0,x_1,\dots, x_n)+G_1(x_1,\dots, x_n)G_2(x_1,\dots, x_n).
		\] 
		A point that satisfies $x_0=G_0=G_1=G_2=0$ is singular, contradiction.
	\end{proof}

	\section{Affine threefolds}
	In this section we are going to study hypersurfaces in $\mathbb{A}^4$, which is the most important part of the proof of our main results. Many tools used in this section have been introduced by Salberger in \cite{salberger}. In particular, we show in Lemma \ref{lem:addhyper} (that is the analogue of \cite[Theorem 2.2]{salberger}) that integral points of bounded height lie on a set of surfaces and we combine this fact with the argument of \cite[Lemma 3.2]{salberger} to bound the cardinality of integral points of bounded height.
	
	Let $Y$ be an affine hypersurface in $\mathbb{A}^4$ defined over $\QQ$ of degree $d$ such that the projective closure of $Y$ is smooth. We say that a prime $p$ is non-singular if the reduction modulo $p$ of the projective closure of $Y$ is smooth. Let $q=p_1\cdot p_2\dots p_k$ be the product of distinct non-singular primes. For each $1\leq i\leq k$, let $P_i$ be a point in $Y(\F_{p_i})$. Let
	\[
	Y(B,P_1,\dots, P_k)=\{\uu{x}\in Y(\ZZ): \uu{x}\equiv P_i \mod p_i, ||\uu{x}||\leq B\},
	\] 
	where with $\uu{x}\equiv P_i \mod p_i$ we mean that $\uu{x}$ reduces to $P_i$ when we reduce $Y$ modulo $p_i$, and with $||\uu{x}||\leq B$ we mean $\abs{x_j}\leq B$ for all $1\leq j\leq 4$. Fix $\ve>0$ and let 
	\begin{equation}\label{eq:defK}
		K=B^{\frac{1}{\sqrt[3]{d}}+\ve}.
	\end{equation}
	\begin{lemma}\label{lem:addhyper}
		There exists a hypersurface $Y(P_1,\dots, P_k)$, that does not contain $Y$, of degree $O_{d,\ve}(K/q+1)$, and such that $Y(B,P_1,\dots, P_k)\subseteq Y(P_1,\dots, P_k) $.
	\end{lemma}
	\begin{proof}
		Assume $q<B^{\frac{1}{\sqrt[3]{d}}+\frac{\ve}{2}}$. Then, by \cite[Theorem 2.2]{salberger}, there exists $Y(P_1,\dots, P_k)$ as in the statement of the lemma with degree 
		\[
		O_d\left(\frac{B^{\frac{1}{\sqrt[3]{d}}}}{q}\log Bq+\log Bq+1\right)=O_{d,\ve}\left(\frac Kq+1\right).
		\]
		If $q\geq B^{\frac{1}{\sqrt[3]{d}}+\frac{\ve}{2}}$, then we can find $Y(P_1,\dots, P_k)$ of degree $O_{d,\ve}(1)$. The proof of this fact is basically identical to the proof of Theorem 15 in \cite{cime}. The only difference is that they are taking $q$ to be a non-singular prime while we are considering $q$ as a product of non-singular primes, but this does not change the proof.
	\end{proof}
	The next lemma is a modification of \cite[Lemma 3.2]{salberger}. We show that integral points of bounded height on an affine threefold lie on a certain set of surfaces or curves. A straightforward application of \cite[Lemma 3.2]{salberger} would produce $O_{d,\ve}(K)$ surfaces of degree $O_d(1)$ and curves of degree $O_{d,\ve}(K^2)$. Our bound produces $O_{d,\ve}(r^2K)$ surfaces of degree $O_d(1)$ and curves of degree $O_{d,\ve}(K^2/r^2)$, for $r$ a fixed prime. In our setting, this leads to a better bound, by choosing an appropriate $r$.
	
	Let $\pi_{\operatorname{sing}}$ be the product of all primes $p$ such that the reduction modulo $p$ of $Y$ is singular.
	\begin{proposition}\label{lem:aff4r}
		Let $Y$ be an affine hypersurface in $\mathbb{A}^4$ defined over $\QQ$ of degree $d$ such that the projective closure of $Y$ is smooth. Assume $\log \pi_{\operatorname{sing}}\ll_d \log B$. Let $d_1,d_2\geq 1$ and fix a non-singular prime $r\ll_{d,\ve}K^{1-\ve}$. Let $Y_{d_1,d_2}$ be
		the complement in $Y$ of the union of all irreducible curves on $Y$, defined over $\QQ$, and of degree at
		most $d_1-1$, and all irreducible surfaces on $Y$, defined over $\QQ$, of degree at
		most $d_2-1$. Then
		\begin{align*}
			N_{\aff}(Y_{d_1,d_2},B)\ll_{d,d_1,d_2,\ve}B^\ve\left(r^2B^{\frac{1}{\sqrt[3]{d}}+\frac{1}{\sqrt{d_2}}+\frac{1}{d_1}}+r^2B^{\frac{1}{\sqrt[3]{d}}+\frac{2}{\sqrt{d_2}}}+\frac{B^{\frac{4}{\sqrt[3]{d}}}}{r}+B^{\frac{3d_1-1+\sqrt[3]{d}}{d_1\sqrt[3]{d}}}\right).
		\end{align*}
	\end{proposition}
	\begin{proof}
		Following \cite[Lemma 3.2, a-b-c]{salberger}, we can find a sequence of distinct primes $p_1,\dots, p_k$ with the following properties:
		\begin{itemize}
			\item For all $i\leq k$, $p_i$ is non-singular and $p_i\neq r$.
			\item For all $i\leq k$, $\log B\leq p_i\ll_{d,\ve} \log B$.
			\item $k\ll_{d,\ve} \log B$.
			\item Let $q=\prod_{1\leq i\leq k} p_i$. Then, 
			\[
			\frac{K}{qr}\ll_{d,\ve} 1
			\]
			and
			\[
			q\ll_{d,\ve} \log B\frac{K}{r}.
			\]
		\end{itemize}
		For all $R\in Y(\F_r)$ and $P_i\in Y(\F_{p_i})$, we fix $Y(R,P_1,\dots,P_k)$ as in Lemma \ref{lem:addhyper}.
		
		Fix $R\in Y(\F_r)$. Define $\{D_\gamma\}_{\gamma\in \Gamma_R}$ as the set of irreducible components of $Y\cap Y(R)$ which are contained in $Y(R, P_1,\dots, P_k)$ for a sequence of points $P_1,\dots, P_k$ with $P_i\in Y(\F_{p_i})$. By Lemma \ref{lem:addhyper}, \[\deg(Y(R, P_1,\dots, P_k))\ll_{d,\ve}\frac{K}{qr}+1\ll_{d,\ve}1\] and so $\deg(D_\gamma)\ll_{d,\ve}1$ for each $\gamma\in \Gamma_R$. Moreover $\Gamma_R$ has cardinality bounded by the cardinality of irreducible components of $Y\cap Y(R)$, that is bounded by \[
		\deg(Y\cap Y(R))\leq \deg(Y)\cdot \deg(Y(R)) \ll_{d,\ve}\frac{K}{r},
		\] applying again Lemma \ref{lem:addhyper}. Define 
		\[
		\Gamma=\cup_{R\in Y(\F_r)}\Gamma_R.
		\]
		We have $\#\Gamma\ll_{d,\ve}\#Y(\F_r)\frac{K}{r}$ and, for each $D_\gamma$ with $\gamma\in \Gamma$, we have $\deg(D_\gamma)\ll_{d,\ve}1$. 
		
		Let $1\leq j\leq k$ and, for all $1\leq i\leq j$, let $P_i$ be a point in $Y(\F_{p_i})$. Define $\gamma(P_1,\dots, P_j)$ as the union of all intersections of distinct irreducible components of $Y\cap Y(R,P_1,\dots, P_j)$ and $Y\cap Y(R,P_1,\dots, P_{j-1})$.
		Notice that $\gamma(R,P_1,\dots, P_j)$ has dimension at most $1$ and, using again Lemma \ref{lem:addhyper}, has degree bounded by
		\begin{align}\label{eq:deggamma}
			&\deg(Y)^2\deg(Y(R,P_1,\dots, P_j))\deg(Y(R,P_1,\dots, P_{j-1}))\nonumber\\&\ll_{d,\ve}\frac{K^2}{r^2\prod_{1\leq i\leq j}p_i\prod_{1\leq i\leq j-1}p_i}.
		\end{align}
		Define $\gamma_{d_1}(R,P_1,\dots, P_j)$ as the variety obtained by removing from $\gamma(R,P_1,\dots, P_j)$ all the dimension $1$ irreducible components of degree at most $d_1-1$ and define
		\begin{equation*}
			W_{d_1}(R,P_1,\dots, P_j)=\left\{\uu{x}\in \gamma_{d_1}(R,P_1,\dots, P_j):
			\begin{array}{l}
				\uu{x}\equiv P_i\mod p_i\, \forall 1\leq i\leq j,\\
				\uu{x}\equiv R\mod r
			\end{array}
			\right\}.
		\end{equation*}
		
		Given ${\uu{x}}\in Y\cap Y(R,P_1,\dots, P_j)$, define $Y_{\uu{x}}(R,P_1,\dots, P_j)$ as the irreducible component of $Y\cap Y(R,P_1,\dots, P_j)$ that contains ${\uu{x}}$.
		Let $\uu{x}\in Y(\ZZ)$ be a point of bounded height, take $P_i\in Y(\F_{p_i})$ such that $\uu{x}$ reduces to $P_i$ modulo $p_i$ and $R\in Y(\F_{r})$ such that $\uu{x}$ reduces to $R$ modulo $r$. If \[Y_{\uu{x}}(R)=Y_{\uu{x}}(R,P_1)=\dots=Y_{\uu{x}}(R,P_1,\dots, P_k),\] then $Y_{\uu{x}}(R)\in \Gamma_R$ (it is contained in $Y(R,P_1,\dots, P_k)$) and so $\uu{x}\in D_{\gamma}$ for some $\gamma\in \Gamma_R$. If not, there is $j\geq 1$ such that $Y_{\uu{x}}(R,P_1,\dots, P_j)\neq Y_{\uu{x}}(R,P_1,\dots, P_{j-1})$. So, $\uu{x}\in 	\gamma(R,P_1,\dots, P_j)$. If $\uu{x}$ does not lie on a curve of degree at most $d_1-1$, then $\uu{x}\in W_{d_1}(R,P_1,\dots, P_j)$.
		Whence, an integral point of bounded height that does not lie on a curve of degree at most $d_1-1$ or on a surface of degree at most $d_2-1$ must lie on $D_\gamma$ for $\gamma\in \Gamma$ and $\deg(D_\gamma)\geq d_2$, or on $W_{d_1}(R,P_1,\dots, P_j)$ for $R\in Y(\F_r)$ and $P_i\in Y(\F_{p_i})$.
		
		Let \[
		N_{\aff}(W_{d_1}(R,P_1,\dots, P_j),B)=\#\{\uu{x}\in W_{d_1}(R,P_1,\dots, P_j): ||\uu{x}||\leq B\}
		\]
		and, given $\gamma\in \Gamma$, let $D_{\gamma,d_1}$ be the complement in $D_\gamma$ of the union of all irreducible curves, defined over $\QQ$, of degree at most $d_1-1$.
		So,
		\begin{align}
			\label{eq:ineqd1d2}
			N_{\aff}(Y_{d_1,d_2},B)\leq& \sum_{\substack{\gamma\in \Gamma\\\deg(D_\gamma)\geq d_2}}N_{\aff}(D_{\gamma,d_1},B)\\&+\sum_{R\in Y(\F_r)}\sum_{1\leq j\leq k}\sum_{\substack{P_i\in Y(\F_{p_i})\\ 1\leq i\leq j}} N_{\aff}(W_{d_1}(R,P_1,\dots, P_j),B)\nonumber.
		\end{align}
		
		By Lemma \ref{lem:surfaces}, the integral points of bounded height on $D_\gamma$ that do not lie on curves of degree smaller than $d_1$ are \[O_{d,d_1,d_2,\ve}\left(B^{\frac{1}{\sqrt{d_2}}+\frac{1}{d_1}+\ve}+B^{\frac{2}{\sqrt{d_2}}+\ve}\right).\] 
		Since $$\#\Gamma\ll_{d,\ve}\#Y(\F_r)\frac{K}{r},$$ we obtain
		\begin{equation}
			\label{eq:ineqsurface}
			\sum_{\substack{\gamma\in \Gamma\\\deg(D_\gamma)\geq d_2}}N_{\aff}(D_{\gamma,d_1},B)\ll_{d,d_1,d_2,\ve}\#Y(\F_r)\left(\frac{K}{r}\right)\left(B^{\frac{1}{\sqrt{d_2}}+\frac{1}{d_1}+\ve}+B^{\frac{2}{\sqrt{d_2}}+\ve}\right).
		\end{equation}
		
		Now, we bound $ N_{\aff}(W_{d_1}(R,P_1,\dots, P_j),B)$. Recall that each $1$-dimensional irreducible component in $W_{d_1}(R,P_1,\dots, P_j)$ has degree at least $d_1$ and the degree of $\gamma(R,P_1,\dots, P_j)$ is 
		$
		O_{d,\ve}\left(\frac{K^2}{d(r,p_1,\dots, p_j)}\right)
		$
		with 
		\[
		d(r,p_1,\dots, p_j)=r^2\prod_{1\leq i\leq j}p_i\prod_{1\leq i\leq j-1}p_i
		\]
		by \eqref{eq:deggamma}. So, by Lemma \ref{lem:poncurves},
		\begin{align*}
			&N_{\aff}(W_{d_1}(R,P_1,\dots, P_j),B)\\
			\ll_{d,d_1,\ve}&B^\ve\left(\frac{K^2}{d(r,p_1,\dots, p_j)}\left(\frac{B}{r\prod_{1\leq i\leq j}p_i}\right)^{\frac{1}{d_1}}+\left(\frac{K^2}{d(r,p_1,\dots, p_j)}\right)^2\right).
		\end{align*}
		
		By \cite[Lemma 2.12]{salberger}, $\prod\#Y(\F_{p_i})\ll_{d}\prod p_i^3$ and then
		\begin{align*}
			&\sum_{\substack{P_i\in Y(\F_{p_i})\\ 1\leq i\leq j}} N_{\aff}(W_{d_1}(R,P_1,\dots, P_j),B)
			\\
			\ll_{d,d_1,\ve}&B^\ve\left(\frac{p_j\prod_{i\leq j} p_iK^2}{r^2}\left(\frac{B}{r\prod_{1\leq i\leq j}p_i}\right)^{\frac{1}{d_1}}+\frac{p_j^2K^4}{r^4\prod_{i\leq j} p_i}\right).
		\end{align*}
		Recalling that $p_i\ll_{d,\ve}\log B$, $\prod p_i\ll_{d,\ve}K\log B/r$, using that $\log B\ll_\ve B^\ve$, and including $B^\ve$ in the definition of $K$ we obtain
		\begin{align*}
			&\sum_{1\leq j\leq k}\sum_{\substack{P_i\in Y(\F_{p_i})\\ 1\leq i\leq j}} N_{\aff}(W_{d_1}(R,P_1,\dots, P_j),B)\\
			\ll_{d,d_1,\ve}&\sum_{1\leq j\leq k}B^\ve\left(\frac{p_j\prod_{i\leq j} p_iK^2}{r^2}\left(\frac{B}{r\prod_{1\leq i\leq j}p_i}\right)^{\frac{1}{d_1}}+\frac{p_j^2K^4}{r^4\prod_{i\leq j} p_i}\right)\\
			\ll_{d,d_1,\ve}&\sum_{1\leq j\leq k}B^\ve\log^2B\left(\frac{(\prod_{i\leq j} p_i)^{1-\frac{1}{d_1}}K^2 B^{\frac{1}{d_1}}}{r^{2+\frac{1}{d_1}}}+\frac{K^4}{r^4}\right)\\
			\ll_{d,d_1,\ve}&\sum_{1\leq j\leq k}\frac{K^{3-\frac{1}{d_1}}}{r^{3}} B^{\frac{1}{d_1}}+\frac{K^4}{r^4}\\
			\ll_{d,d_1,\ve}&k\left(\frac{K^{3-\frac{1}{d_1}}}{r^{3}} B^{\frac{1}{d_1}}+\frac{K^4}{r^4}\right).
		\end{align*}
		Recall $k\ll_{d,\ve}\log B$ and so
		\begin{align*}
			&\sum_{R\in Y(\F_r)}\sum_{1\leq j\leq k}\sum_{\substack{P_i\in Y(\F_{p_i})\\ 1\leq i\leq j}} N_{\aff}(W_{d_1}(R,P_1,\dots, P_j),B)\\\ll_{d,d_1,\ve}&
			\sum_{R\in Y(\F_r)}\log B\left(\frac{K^{3-\frac{1}{d_1}}}{r^{3}} B^{\frac{1}{d_1}}+\frac{K^4}{r^4}\right)\\\ll_{d,d_1,\ve}& \#Y(\F_r)\left(\frac{K^{3-\frac{1}{d_1}}}{r^{3}} B^{\frac{1}{d_1}}+\frac{K^4}{r^4}\right).
		\end{align*}
		By \eqref{eq:ineqd1d2} and \eqref{eq:ineqsurface},
		\begin{equation}\label{eq:endaff}
			N_{\aff}(Y_{d_1,d_2},B)\ll_{d,d_1,d_2,\ve}\#Y(\F_r)\left(\frac{KB^{\frac{1}{\sqrt{d_2}}+\frac{1}{d_1}}}{r}+\frac{KB^{\frac{2}{\sqrt{d_2}}}}{r}+\frac{K^{3-\frac{1}{d_1}}}{r^{3}} B^{\frac{1}{d_1}}+\frac{K^4}{r^4}\right)
		\end{equation}
		and we conclude using that $\#Y(\F_r)\ll_{d}r^3$.
	\end{proof}
	\begin{remark}
		Taking $r=O_{d,\ve}(B^\ve)$, Proposition \ref{lem:aff4r} would be identical to \cite[Lemma 3.2]{salberger}. Taking $r$ as large as possible, say $r\approx O_{d,\ve}(K^{1-\ve})$, Proposition \ref{lem:aff4r} could be proved using Heath-Brown's determinant method \cite{cime}. 
	\end{remark}
	Let $N_{\aff,4}(Y,B)$ be the cardinality of the set of integral points of bounded height in $Y$ that do not lie on curves of degree at most $3$ on $Y$. Recall that $\theta(d)$ is defined in \eqref{eq:theta}.
	\begin{proposition}\label{lem:finalaff4r}
		Let $Y$ be an affine hypersurface defined over $\QQ$ of degree $d\geq 6$ in $\mathbb{A}^4$ such that the projective closure of $Y$ is smooth. Then, for all $\ve>0$,
		\[
		N_{\aff,4}(Y,B)\ll_{d,\ve} B^{\theta(d)+1+\ve}.
		\]
	\end{proposition}
	\begin{proof}
		Assume that the integral points of bounded height on $Y$ are contained in another threefold of degree at most $d$. Then, the points of bounded height lie on a surface of degree at most $d^2$. By Lemma \ref{lem:NF}, there are no surfaces on $Y$ of degree less than $d$. The integral points of bounded height on an irreducible surface of degree at least $d$ and at most $d^2$ that do not lie on curves of degree at most $3$ are
		\[
		O_{d,\ve}\left(B^{\frac{1}{\sqrt{d}}+\frac{1}{4}+\ve}+B^{\frac{2}{\sqrt{d}}+\ve}\right)
		\]
		by Lemma \ref{lem:surfaces} and the proposition easily follows. So, we can assume that the integral points of bounded height on $Y$ are not contained in another threefold of degree at most $d$, thus we can apply \cite[Lemma 3.1 (b)]{salberger}. Hence, $\log \pi_{\operatorname{sing}}\ll_{d,\ve}\log B$ and we can apply Proposition \ref{lem:aff4r}.
		
		By Lemma \ref{lem:NF}, there are no surfaces on $Y$ of degree less than $d$.
		So,
		\[
		N_{\aff,4}(Y,B)\ll_{d,\ve}N_{\aff}(Y_{4,d},B).
		\]
		Applying Proposition \ref{lem:aff4r} with $d_1=4$ and $d_2=d$ we obtain
		\[
		N_{\aff}(Y_{4,d},B)\ll_{d,\ve}B^\ve\left(r^2B^{\frac{1}{\sqrt[3]{d}}+\frac{1}{\sqrt{d}}+\frac{1}{4}}+r^2B^{\frac{1}{\sqrt[3]{d}}+\frac{2}{\sqrt{d}}}+\frac{B^{\frac{4}{\sqrt[3]{d}}}}{r}+B^{\frac{11}{4\sqrt[3]{d}}+\frac{1}{4}}\right).
		\]
		If $d\geq 125$, we can take by \cite[Lemma 3.1]{salberger} a non-singular prime $r\ll_{d,\ve}B^{\ve}$, and we obtain
		\begin{align}\label{eq:125}
			N_{\aff}(Y_{4,d},B)&\ll_{d,\ve}B^\ve\left(B^{\frac{1}{\sqrt[3]{d}}+\frac{1}{\sqrt{d}}+\frac{1}{4}}+B^{\frac{1}{\sqrt[3]{d}}+\frac{2}{\sqrt{d}}}+B^{\frac{4}{\sqrt[3]{d}}}+B^{\frac{11}{4\sqrt[3]{d}}+\frac{1}{4}}\right)\nonumber\\&
			\ll_{d,\ve}B^{\frac{11}{4\sqrt[3]{d}}+\frac{1}{4}+\ve}.
		\end{align}
		If $20\leq d<125$, take $r$ a non-singular prime such that
		\[
		B^{\frac{5}{4\sqrt[3]{d}}-\frac{1}{4}}<r\ll_{d,\ve}B^{\frac{5}{4\sqrt[3]{d}}-\frac{1}{4}},
		\]
		that exists by \cite[Lemma 3.1]{salberger}. Notice $r\ll_{d,\ve}K^{1-\ve}$, so we can apply Proposition \ref{lem:aff4r}. Thus,
		\[
		N_{\aff}(Y_{4,d},B)\ll_{d,\ve}B^{\frac{7}{2\sqrt[3]{d}}+\frac{1}{\sqrt{d}}-\frac{1}{4}+\ve}+B^{\frac{7}{2\sqrt[3]{d}}+\frac{2}{\sqrt{d}}-\frac{1}{2}+\ve}+B^{\frac{11}{4\sqrt[3]{d}}+\frac{1}{4}+\ve}.
		\]
		Since $20\leq d$, by direct computation
		\begin{align*}
			N_{\aff}(Y_{4,d},B)&\ll_{d,\ve}B^{\frac{7}{2\sqrt[3]{d}}+\frac{1}{\sqrt{d}}-\frac{1}{4}+\ve}+B^{\frac{7}{2\sqrt[3]{d}}+\frac{2}{\sqrt{d}}-\frac{1}{2}+\ve}+B^{\frac{11}{4\sqrt[3]{d}}+\frac{1}{4}+\ve}\ll_{d,\ve}B^{\frac{11}{4\sqrt[3]{d}}+\frac{1}{4}+\ve}
		\end{align*}	
		and so
		\[
		N_{\aff,4}(Y,B)\ll_{d,\ve}B^{\frac{11}{4\sqrt[3]{d}}+\frac{1}{4}+\ve}.
		\]
		Combining this with \eqref{eq:125}, we get that for $d\geq 20$, it holds
		\begin{equation}\label{eq:20}
			N_{\aff,4}(Y,B)\ll_{d,\ve}B^{\frac{11}{4\sqrt[3]{d}}+\frac{1}{4}+\ve}.
		\end{equation}
		Notice in particular that, for $d\geq 50$, we have $N_{\aff,4}(Y,B)\ll_{d,\ve} B^{1+\ve}$.
		If $16> d\geq 9$, we conclude by taking
		\[
		B^{\frac{1}{\sqrt[3]{d}}-\frac{2}{3\sqrt{d}}}<r\ll_{d,\ve}B^{\frac{1}{\sqrt[3]{d}}-\frac{2}{3\sqrt{d}}},
		\] 
		and if $20> d\geq 16$, we conclude by taking
		\[
		B^{\frac{1}{\sqrt[3]{d}}-\frac{1}{3\sqrt{d}}-\frac{1}{12}}<r\ll_{d,\ve}B^{\frac{1}{\sqrt[3]{d}}-\frac{1}{3\sqrt{d}}-\frac{1}{12}}.
		\] 
		
		It remains to deal with the case $6\leq d\leq 8$. In this case, we use a different method that leads to a better bound. Assume that $Y$ is defined by $f(x_1,\dots ,x_4)=0$ and define $Y_b=Y\cap \{x_4=b\}$, that is an irreducible surface of degree $d$ by Lemma \ref{lem:irr}. If the projective closure of $Y_b$ is not smooth, the integral points of bounded height are $O_{d,\ve}(B^{1+\frac{1}{d}+\ve})$ by \cite[Theorem A]{pilaaff}. By Lemma \ref{lem:B}, we can assume that there are $O_d(1)$ values of $b$ for which the projective closure of $Y_b$ is not smooth. Assume now that the projective closure of $Y_b$ is smooth. By \cite[Appendix]{HBannals}, there are $O_d(1)$ irreducible curves of degree $\leq d-2$ on $Y_b$. Hence, by Lemma \ref{lem:surfaces}, the integral points outside curves of degree at most $3$ on $Y_b$ are
		\[
		O_{d,\ve}\left(B^{\frac{1}{4}+\ve}+B^{\frac{1}{\sqrt{d}}+\frac{1}{d-2}+\ve}+B^{\frac{2}{\sqrt{d}}+\ve}\right)\ll_{d,\ve}B^{\frac{2}{\sqrt{d}}+\ve}
		\]
		and so
		\[
		N_{\aff,4}(Y,B)=\sum_{\abs{b}\leq B}N_{\aff,4}(Y_b,B)\ll_{d,\ve}B^{1+\frac{2}{\sqrt{d}}+\ve}.
		\qedhere
		\]
	\end{proof}
	Notice that the proof of the previous Proposition holds also for $d=5$ with $\theta(d)=2/\sqrt{d}$.
	\begin{remark}
		We briefly compare the bound of Proposition \ref{lem:finalaff4r} with the bounds that one would obtain with different methods. If we slice $Y$ with $O(B)$ hyperplanes (by fixing one variable) and then apply the determinant method we would get \[N_{\aff,4}(Y,B)\ll_{d,\ve}B^{\frac{5}{4}+\ve}+B^{1+\frac{2}{\sqrt{d}}+\ve},\] see the end of the proof of Proposition \ref{lem:finalaff4r}. This bound is worse than the one in Proposition \ref{lem:finalaff4r} for all $d\geq 9$. 
		
		One could also apply Heath-Brown's determinant method \cite{cime}. 
		The integral points of bounded height on $Y$ are contained in $O_{d,\ve}(B^{\frac{3}{\sqrt[3]{d}}})$ surfaces of degree $O_\ve(1)$. 
		After removing the points on curves of degree less than $3$, the integral points of bounded height on each surface are $O_{d,\ve}(B^{\frac{2}{\sqrt{d}}+\ve}+B^{\frac{1}{4}+\frac{1}{\sqrt{d}}+\ve})$ and so we would get
		\[
		N_{\aff,4}(Y,B)\ll_{d,\ve}B^{\frac{3}{\sqrt[3]{d}}+\frac{2}{\sqrt{d}}+\ve}+B^{\frac{3}{\sqrt[3]{d}}+\frac{1}{4}+\frac{1}{\sqrt{d}}+\ve}.
		\]Also this bound is worse than the one in Proposition \ref{lem:finalaff4r} for all $d\geq 6$. 
	\end{remark}
	Now, we prove the analogue of Theorem \ref{thm:mainaff} for hypersurfaces in $\mathbb{A}^4$. The proof of Theorem \ref{thm:mainaff}, and also of Theorem \ref{thm:main}, will follow from the next lemma.
	\begin{lemma}\label{rem:counterexample}
		Let $Y$ be an affine hypersurface defined over $\QQ$ of degree $d\geq 6$ in $\mathbb{A}^4$ such that the projective closure of $Y$ is smooth. Then, for all $\ve>0$,
		\[
		N_{\aff}(Y,B)\ll_{d,\ve} B^{\theta(d)+1+\ve},
		\]
		unless the following hold. After a linear change of variables, $Y$ is defined by an equation of the form
		\[
		f(x_1,x_2,x_3,x_4)=f_0(x_1,x_2)+x_3g(x_1,x_2,x_3,x_4).
		\]
		Let $C\subseteq\mathbb{A}^2$ be the curve defined by $f_0(x_1,x_2)=0$. Then, $C$ has degree $d$ and its projective closure is smooth. In this case,
		\[
		N_{\aff}(Y,B)\ll_{d,\ve} B^{\theta(d)+1+\ve}+BN_{\aff}(C,B).
		\]
		In particular,
		\[
		N_{\aff}(Y,B)\ll_{d,\ve} B^{\theta(d)+1+\ve}+B^{1+\frac 1d+\ve}.
		\]
	\end{lemma}
	\begin{proof}
		By Proposition \ref{lem:finalaff4r}, we just need to bound points on curves of degree at most $3$. Let $W$ be an irreducible component of the reduced scheme structure of the Hilbert's scheme of points on curves of degree at most $3$, that has degree $O_d(1)$. As in the proof of Lemma \ref{lem:remaininglinesaff}, $W$ cannot have dimension $3$ and if it has dimension $\leq 1$ we conclude by the trivial bound \eqref{eq:trivialaff}. So, we assume $\dim W=2$.
		If $W$ is not cylindrical over a curve, we would conclude that 
		\[
		N_{\aff}(Y,B)\ll_{d,\ve} B^{\theta(d)+1+\ve},
		\]
		following the proof of Lemma \ref{lem:remaininglinesaff}.
		Assume that $W$ is cylindrical over a curve. Thus, after a linear change of variables, we can assume that $W=V(f_i(x_1,x_2,x_3))_{1\leq i\leq s}$. Assume that $Y$ is defined by $f=0$ with
		\[
		f(x_1,x_2,x_3,x_4)=\sum_{i=0}^d x_4^ih_i(x_1,x_2,x_3).
		\]
		Since $W\subseteq Y$, then $W\subseteq \{h_i=0\}$ for all $0\leq i\leq d$. Hence, $h_d=0$. If $h_{d-1}=0$, then the point $[x_0,x_1,x_2,x_3,x_4]=[0,0,0,0,1]$ in the projective closure of $Y$ would be singular, contradiction. So, $h_{d-1}$ must have degree $1$ and, after a linear change of variables, we can assume $h_{d-1}(x_1,x_2,x_3)=x_3$. So, $W\subseteq \{x_3=0\}$ and then $W=V(f_0(x_1,x_2),x_3)$. Since $W$ has degree at least $d$ by Lemma \ref{lem:NF}, it follows that $f_0$ has degree at least $d$. For all $i\geq 1$, we have $\deg(h_i)<d$ and then $x_3\mid h_i$. Thus,
		\[
		f(x_1,x_2,x_3,x_4)=f_0(x_1,x_2)h_0'(x_1,x_2)+x_3g(x_1,x_2,x_3,x_4).
		\]
		Since $\deg(f_0)\geq d$, we have $h_0'$ is a constant, and it is not zero since otherwise $Y$ would not be irreducible. Hence,
		\[
		f(x_1,x_2,x_3,x_4)=f_0(x_1,x_2)+x_3g(x_1,x_2,x_3,x_4)
		\]
		with $f_0$ an irreducible polynomial of degree $d$. Since the projective closure of $Y$ is smooth, we have that the projective closure of $C$, the affine curve defined by $f_0(x_1,x_2)=0$, is smooth. Notice in particular that the genus of $C$ is strictly larger than $1$. We have $N_{\aff}(W,B)=(2B+1)N_{\aff}(C,B)$ and so
		\[
		N_{\aff}(Y,B)\ll_{n,d,\ve} B^{\theta(d)+1+\ve}+ BN_{\aff}(C,B).
		\]
		Finally, the last line of the statement follows from Lemma \ref{lem:poncurves}.
	\end{proof}
	\section{General case}
	The goal of this section is to prove Theorems \ref{thm:main} and \ref{thm:mainaff}. The cases for $n\geq 5$ will follow from the affine case for $n=4$. We will reduce to that case by slicing our hypersurface with hyperplanes. This slicing procedure is similar to the one done in \cite[Section 3]{BHB} by Browning and Heath-Brown.
	
	The projective case for $n=4$, which we will establish in the next lemma, follows directly from Salberger's work in \cite[Section 3]{salAENS}.
	This result should not be regarded as original, as it merely involves a straightforward application of Salberger’s techniques. However, since it does not appear in the literature and we aim to cover all cases, we have chosen to include it.
	Replicating the approach from the previous section in the projective setting would lead to a weaker bound.
	
	\begin{lemma}\label{lem:finalproj4r}
		Let $X$ be a smooth projective hypersurface defined over $\QQ$ of degree $d\geq 6$ in $\mathbb{P}^4$. Then, for all $\ve>0$,
		\[
		N(X,B)\ll_{d,\ve}B^{2+\ve}+B^{\frac{45}{16\sqrt{d}}+\frac{5}{4}+\ve}.
		\]
	\end{lemma}
	\begin{proof}
		This is basically \cite[Corollary 3.9]{salAENS}, but using Salberger's new version of the determinant method. So, we are just going to sketch the proof. The rational points of bounded height that lie on irreducible curves of degree at most $3$ are at most $O_{d,\ve}(B^{2+\ve})$ by Lemma \ref{lem:remlines}. 
		By Siegel's Lemma, the rational points of bounded height lie on a set of $O(B^{\frac{5}{4}})$ hyperplanes. We call the intersection of $X$ with one of these hyperplanes a hyperplane section. Given a smooth hyperplane section, the rational points of bounded height that lie on the smooth hyperplane section and not on irreducible curves of degree $\leq 3$ have cardinality
		\[
		O_{d,\ve}\left(B^{\frac{3}{\sqrt{d}}+\ve}+B^{\frac{1}{2}+\ve}\right)
		\]
		by Lemma \ref{lem:projsurf}. Since there are $O(B^{\frac{5}{4}})$ hyperplanes sections, rational points of bounded height that lie on smooth hyperplane sections and not on irreducible curves of degree $\leq 3$ have cardinality 
		\begin{equation}\label{eq:smoothsurf}
			O_{d,\ve}\left(B^{\frac{15}{16}\cdot\frac{3}{\sqrt{d}}+\frac{5}{4}+\ve}+B^{\frac{15}{16}\cdot\frac{1}{2}+\frac{5}{4}+\ve}\right)=O_{d,\ve}\left(B^{\frac{45}{16\sqrt{d}}+\frac{5}{4}+\ve}+B^{\frac{55}{32}+\ve}\right).
		\end{equation}
		The extra saving given by the multiplication by $15/16$ in the exponent can be obtained replicating \cite[Proof of Theorem 3.3]{salberger3fold}, see also \cite[Section 5]{salAENS}. 
		As is shown in the proof of Theorem 3.6 in \cite{salAENS}, rational points of bounded height lie on a set of singular hyperplane sections and this set has cardinality 
		\[
		O_{d,\ve}\left(B^{\frac{1}{(d-1)d^{1/3}}+\frac 12}\right).
		\]
		Thus, the rational points of bounded height that lie on geometrically reducible hyperplane sections and not on irreducible curves of degree $\leq 3$ have cardinality
		\[
		O_{d,\ve}\left(B^{\frac{1}{(d-1)d^{1/3}}+\frac{3}{2\sqrt{d}}+1+\ve}\right)
		\]
		by Lemma \ref{lem:projsurf}. We conclude by noticing that
		$$
		\frac{1}{(d-1)d^{1/3}}+\frac{3}{2\sqrt{d}}+1\leq \frac{45}{16\sqrt{d}}+\frac{5}{4},
		$$
		for $d\geq 2$.
	\end{proof} 
	
	Now, we are going to prove Theorem \ref{thm:mainaff} for $n=5$. For $n\geq 6$, we are going to prove Theorem \ref{thm:mainaff} by induction and this will be the base case.	
	Recall that $\theta(d)$ is defined in \eqref{eq:theta} and $\theta(d)\geq0$ for all $d\geq 6$.
	\begin{lemma}\label{lem:case5}
		Let $d\geq 6$.
		Let $Y$ be an affine hypersurface defined over $\QQ$ of degree $d$ in $\mathbb{A}^5$ such that the projective closure of $Y$ is smooth. Then, for all $\ve>0$,
		\[
		N_{\aff}(Y,B)\ll_{d,\ve} B^{2+\theta(d)+\ve}.
		\]
	\end{lemma}
	\begin{proof}
		First of all, we bound integral points of bounded height that do not lie on curves of degree at most $3$.
		Consider the affine hypersurface $Y$ defined by $f(x_1,\dots, x_5)=0$. Let $Y_b$ be the affine variety defined by $f(x_1,\dots, b)=0$ and let $\mathcal{B}$ be the set of integers such that the projective closure of $Y_b$ is singular or it has degree smaller than $d$. By Lemma \ref{lem:B}, we can assume that $\#\mathcal{B}\ll_{d,n}1$. If $b$ does not belong to $\mathcal{B}$, then the projective closure of $Y_b$ is smooth (and notice it has degree $d$) and we can apply Proposition \ref{lem:finalaff4r} to bound the points on $Y_b$. Fix now $b\in \mathcal{B}$.
		Notice that $Y_b$ is irreducible, of degree $d$, and not cylindrical over a curve by Lemmata \ref{lem:cylsing} and \ref{lem:irr}. By \eqref{eq:affdimgrowth} the integral points of bounded height on $Y_b$ are at most $B^{2+\ve}$. Therefore, by Proposition \ref{lem:finalaff4r},
		\[
		N_{\aff,4}(Y,B)\leq \sum_{0\leq \abs{b}\leq B} N_{\aff,4}(Y_b,B)\ll_{d,\ve} B^{2+\ve}+B\cdot B^{\theta(d)+1+\ve}.
		\]
		So, $N_{\aff,4}(Y,B)\ll_{d,\ve} B^{2+\theta(d)+\ve}$, since $\theta(d)\geq 0$. It remains to deal with points that lie on curves of degree at most $3$. If $d\geq 7$, we conclude by Lemma \ref{lem:remaininglinesaff} (notice $d>n+1$). If $d=6$, we repeat the induction above to get
		\[
		N_{\aff}(Y,B)\ll_{\ve} B^{2+\ve}+ \sum_{\substack{0\leq \abs{b}\leq B\\b\notin \mathcal{B}}} N_{\aff}(Y_b,B)
		\]
		and so, by Lemma \ref{rem:counterexample}, 
		\[
		N_{\aff}(Y,B)\ll_{\ve}B^{2+\ve}+B\cdot B^{\theta(6)+1+\ve}+B^{2+\frac{1}{6}+\ve}.
		\]
		We conclude by noticing that $\theta(6)+2\geq 2+1/6$.
	\end{proof}
	\begin{lemma}\label{lem:indaff}
		Let $d\geq 6$ and $n\geq 5$.
		Let $Y$ be an affine hypersurface defined over $\QQ$ of degree $d$ in $\mathbb{A}^n$ such that the projective closure of $Y$ is smooth. Then, for all $\ve>0$,
		\[
		N_{\aff}(Y,B)\ll_{n,d,\ve} B^{n-3+\theta(d)+\ve}.
		\]
	\end{lemma}
	\begin{proof}
		We prove the lemma by induction. We fix $n\geq 5$ and assume that the lemma has been proved for all $5\leq m<n$ (the case $m=5$ is done in Lemma \ref{lem:case5}).
		With an argument identical to the one in Lemma \ref{lem:case5}, we get
		\[
		N_{\aff}(Y,B)\leq \sum_{0\leq \abs{b}\leq B} N_{\aff}(Y_b,B)\ll_{n,d,\ve} B^{n-3+\ve}+B\cdot B^{n-4+\theta(d)+\ve}
		\]
		and easily conclude.
	\end{proof}
	\begin{lemma}\label{lem:indproj}
		Let $d\geq 6$ and $n\geq 5$.
		Let $X$ be a smooth projective hypersurface defined over $\QQ$ of degree $d$ in $\mathbb{P}^n$. Then, for all $\ve>0$,
		\[
		N(X,B)\ll_{n,d,\ve} B^{n-2+\theta(d)+\ve}.
		\]
	\end{lemma}
	\begin{proof}
		Assume that $X$ is defined by $F(x_0,x_1,\dots,x_n)=0$. Let $f_b(x_1,\dots,x_n)=F(b,x_1,\dots,x_n)$ for $b\neq 0$. If $\deg f_b<d$, then $x_0\mid F$ and so $X$ would be singular. So, $\deg f_b=d$ and the projective closure of $f_b$ is smooth. For $b\neq 0$, let $Y_b=X\cap \{x_0=b\}$, which is an affine hypersurface. Whence,
		\[
		N(X,B)\leq N(\{F=0\}\cap \{x_0=0\},B)+\sum_{1\leq \abs{b}\leq B} N_{\aff}(Y_b,B).
		\]
		By \eqref{eq:dimgrowth} and Lemma \ref{lem:indaff}, 
		\[
		N(\{F=0\}\cap \{x_0=0\},B)\ll_{n,d,\ve} B^{n-2+\ve}
		\]
		and 
		\[
		N_{\aff}(Y_b,B)\ll_{n,d,\ve} B^{n-3+\theta(d)+\ve}.
		\]
		Therefore, we conclude by recalling $\theta(d)\geq 0$.
	\end{proof}
	The proof of Theorems \ref{thm:main} and \ref{thm:mainaff} is now completed, we summarise it. The projective case for $n=4$ is done in Lemma \ref{lem:finalproj4r}. The affine case for $n\geq 5$ is done in Lemma \ref{lem:indaff}. The projective case for $n\geq 5$ is done in Lemma \ref{lem:indproj}.

\end{document}